\documentclass[11pt]{amsart}

\usepackage{amssymb, amsmath, nccmath, color, hyperref, url, amsxtra,  fullpage}
\usepackage{placeins}
\usepackage{float}
\usepackage{mathtools}
\usepackage{pdfpages}
\usepackage{graphicx}
\usepackage{mathtools}

\usepackage{tikz}
\usepackage{geometry}
\usepackage{textcomp}
\newcommand{\boxWithLabel}[4]
{
	\draw (#2,#3) -- (#2-#1,#3) -- (#2-#1,#3-#1) -- (#2,#3-#1) -- (#2,#3);
	\node (a) at (#2-0.5*#1,#3-0.5*#1) {#4};
}

\allowdisplaybreaks

\newtheorem{theorem}{Theorem}[section]
\newtheorem{corollary}[theorem]{Corollary}
\newtheorem{lemma}[theorem]{Lemma}

\theoremstyle{definition}
\theoremstyle{conjecture}
\newtheorem{defn}[theorem]{Definition}

\usepackage{tikz}
\newcounter{x}
\newcounter{y}
\newcounter{z}

\begin{document}

\title{Congruences for Restricted Plane Overpartitions Modulo 4 and 8}

\author{ Ali H. Al-Saedi}
\address{Ali H. Al-Saedi}
\address{Oregon State University, Corvallis, OR 97331, USA}
\email{alsaedia@math.oregonstate.edu\\alsaedia@lifetime.oregonstate.edu}

\begin{abstract}
In 2009, Corteel, Savelief and Vuleti\'c generalized the concept of overpartitions to a new object called plane overpartitions. In recent work, the author considered a restricted form of plane overpartitions called $k$-rowed plane overpartions and proved a method to obtain congruences for these and other types of combinatorial generating functions. In this paper, we prove several restricted and unrestricted plane overpartition congruences modulo $4$ and $8$ using other techniques. 
\end{abstract}



\maketitle


\section{\textbf{Introduction and Statement of Results}}
\subsection{Partitions and Overpartitions}

A {\it{partition}} of a positive integer $n$ is a nonincreasing sequence of positive integers that sum to $n$. The total number of partitions of $n$ is denoted by $p(n)$. One can also consider partitions where the parts are restricted to a specific set $S$ of integers and let $p(n;S)$ denote the number of partitions of $n$ into parts from $S$.
For example, consider the set
$$S=\{1,2,5,8\}.$$ Then $p(5;S)=4$ since the partitions of $5$ with parts from $S$ are
$$5,2+2+1, 2+1+1+1,1+1+1+1+1.$$ 
The generating function for this type of  partition is given by
\begin{align}\label{set}
\sum_{n=0}^{\infty}{p(n;S)q^n}=\prod_{n\in S}\frac{1}{1-q^n}.
\end{align}

Note that $S$ can be a multiset where repeated numbers are treated independently. For example, let $S=\{1,2_{1},2_{2},3_{1},3_{2}\}$, then we have the following partitions of $n=4$ into parts from $S$,
$$3_{1}+1,3_{2}+1,2_{1}+2_{1},2_{1}+2_{2},2_{2}+2_{2},2_{1}+1+1,2_{2}+1+1,1+1+1+1.$$   
Thus,  $p(4;S)=8.$ Note that repeated numbers in a multiset are given an intrinsic order in terms of their subscript.

An {\it{overpartition}} of a positive integer $n$ is a partition of $n$ in which the first occurrence of a part may be overlined. We denote the number of overpartitions of $n$ by $\overline{p}(n)$ and define $\overline{p}(0):=1$. 

For example, when $n=3$ we see that $\overline{p}(n)=8,$ with overpartitions given by
$$3,\overline{3}, 2+1, \overline{2}+1, 2+\overline{1}, \overline{2}+\overline{1}, 1+1+1, \overline{1}+1+1.$$

An overpartition can be interpreted as a pair of partitions one into distinct parts corresponding with the overlined parts and the other unrestricted. Thus, we see that the generating function for overpartitions is given by
\begin{equation}
\overline{P}(q):=\sum_{n=0}^{\infty}\overline{p}(n)q^{n}=\prod_{n=1}^{\infty} \frac{1+q^n}{1-q^n}= 1+2q+4q^{2}+8q^{3}+14q^{4}+ \cdots.
\end{equation}

 Overpartitions have been studied extensively by  Corteel, Lovejoy, Osburn,  Bringmann, Mahlburg, Hirschhorn, Sellers, and many other mathematicians. For example, see \cite{bringmann2008rank}, \cite{corteel2004overpartitions}, \cite{hirschhorn2005arithmetic},  \cite{hirschhorn}, \cite{hirschhorn2006arithmetic}, \cite{lovejoy2003gordon}, \cite{lovejoy2004overpartition}, \cite{lovejoy2008rank} and \cite{mahlburg2004overpartition} to mention a few.

The well-known Jacobi triple product identity \cite{andrews1965simple} is given by
\begin{align}\label{jacobi1}
\prod_{n=1}^{\infty}{(1-q^{2n})(1+zq^{2n-1})(1+z^{-1} q^{2n-1})}=\sum_{n=-\infty}^{\infty}z^n q^{n^2},
\end{align}
which converges when $z\not =0$ and $|q|<1$. Letting $z=1$ in \eqref{jacobi1}, one can observe one of Ramanujan's classical theta functions
\begin{align}\label{mo1}
\phi(q):=\sum_{n=-\infty}^{\infty}{q^{n^2}}=\prod_{n=1}^{\infty}{(1-q^{2n})}{(1+q^{2n-1})^2}.
\end{align}
Replacing $q$ by $-q$ in \eqref{mo1}, we get
\begin{align}\label{mo2}
\phi(-q)=\sum_{n=-\infty}^{\infty}{(-1)^n q^{n^2}}=\prod_{n=1}^{\infty}{(1-q^{2n})}{(1-q^{2n-1})^2}
=\prod_{n=1}^{\infty}{\frac{(1-q^{n})}{(1+q^{n})}}=\frac{1}{\overline{P}(q)}.
\end{align}
Note that $\phi(q)$ can be written as
\begin{align*}
\phi(q)=1+2\sum_{n=1}^{\infty}q^{n^2}.
\end{align*}
Thus, the generating function of overpartitions has the following $2$-adic expansion,
\begin{align}\label{GenParOne}
\overline{P}(q)=&\frac{1}{\phi(-q)}=\frac{1}{1+2\sum_{n=1}^{\infty}(-1)^{n^2}q^{n^2}}\nonumber\\
=&1+\sum_{k=1}^{\infty}{2^k(-1)^k\left(\sum_{n=1}^{\infty}(-1)^{n^2}q^{n^2}\right)^k}\nonumber\\
=&1+\sum_{k=1}^{\infty}{2^k}\sum_{n_{1}^2+\dots+n_{k}^2=n}{(-1)^{n+k}q^n}\nonumber\\
=&1+\sum_{k=1}^{\infty}{2^k}\sum_{n=1}^{\infty}{(-1)^{n+k}c_{k}(n)q^n},
\end{align}
where $c_{k}(n)$ denotes the number of representations of $n$ as a sum of $k$ squares of positive
integers. 

 Several overpartition congruences modulo small powers of $2$ have been found using the $2$-adic expansion formula \eqref{GenParOne}.  For example, Mahlburg \cite{mahlburg2004overpartition} proves that 
\begin{align*}
\{n\in \mathbb{N}\;|\;\overline{p}(n)\equiv 0\pmod{64}\}
\end{align*}
is a set of density 1\footnote{The sequence $A$ of positive integers $a1<a2<\cdots$  has a density $\delta(A)$ if ${\delta}(A)=\underset{n \to \infty}{\lim} \frac{A(n)}{n}.$ For more details about arithmetic density of integers, one may see \cite{niven1951asymptotic}.
}.


Later, Kim \cite{kim2008overpartition} generalized Mahlburg's result modulo $128$. Furthermore, Mahlburg conjectures \cite{mahlburg2004overpartition} that for any integer $k\geq 1$,
\begin{align*}
\overline{p}(n)\equiv 0\pmod{2^k},
\end{align*}
for almost all integers $n$.

\subsection{Plane Partitions and Plane Overpartitions}
As a natural generalization of partitions, MacMahon \cite{andrews}  defines a {\it{plane partition}} of $n$ as a two dimensional array $\pi=(\pi_{ij})_{i,j \geq 1 }$ of nonnegative integers $\pi_{ij}$, with $i$ indexing rows and $j$ indexing columns, that are weakly decreasing in both rows and columns and for which $|\pi|:=\sum{\pi_{ij}}=n<\infty$.

 Corteel,  Savelief and Vuleti{\'c} \cite{corteelplane} define {\it{plane overpartitions}} as a generalization of the overpartitions  as follows.
 \begin{defn}[Corteel, Savelief, Vuleti{\'c}, \cite{corteelplane}]
 A plane overpartition is a plane partition where 
 \begin{enumerate}
 	\item
 	in each row the last occurrence of an integer can be overlined or not and all the other occurrences of this integer in the row are not overlined and,
 	\item
 	 in each column the first occurrence of an integer can be overlined or not and all the other occurrences of this integer in the column are overlined.
 \end{enumerate}
 \end{defn}  
 
 Plane overpartitions can be  represented  in the form of Ferrers-Young diagrams. For example, 
$$
{
	\begin{tikzpicture}[scale=1]

	\boxWithLabel{0.5}{8}{4}{5}
	\boxWithLabel{0.5}{8.5}{4}{4}
	\boxWithLabel{0.5}{9}{4}{$\bar{4}$}
	\boxWithLabel{0.5}{9.5}{4}{3}
	\boxWithLabel{0.5}{10}{4}{$\bar{1}$}

	\boxWithLabel{0.5}{8}{3.5}{3}
	\boxWithLabel{0.5}{8.5}{3.5}{2}
	\boxWithLabel{0.5}{9}{3.5}{1}
	
	\boxWithLabel{0.5}{8}{3}{2}
	\boxWithLabel{0.5}{8.5}{3}{$\bar{2}$}
	\boxWithLabel{0.5}{9}{3}{$\bar{1}$}
	
	\boxWithLabel{0.5}{8}{2.5}{1}
	\boxWithLabel{0.5}{8.5}{2.5}{1}
	
	\boxWithLabel{0.5}{8}{2}{$\bar{1}$}
\end{tikzpicture}
}
$$
is a plane overpartition of $31$.



The total number of plane overpartitions of $n$ is denoted by $\overline{pl}(n)$.
For example, there are $16$ plane overpartitions for $n=3$ are as follows,
$$
{
	\begin{tikzpicture}[scale=1]
	
	\boxWithLabel{0.5}{0}{4}{3}
	\boxWithLabel{0.5}{1}{4}{$\bar{3}$}
	
	\boxWithLabel{0.5}{2}{4}{2}
	\boxWithLabel{0.5}{2.5}{4}{1}

	\boxWithLabel{0.5}{3.5}{4}{$\bar{2}$}
	\boxWithLabel{0.5}{4}{4}{1}

	\boxWithLabel{0.5}{5}{4}{2}
	\boxWithLabel{0.5}{5.5}{4}{$\bar{1}$}
	
	\boxWithLabel{0.5}{6.5}{4}{$\bar{2}$}
	\boxWithLabel{0.5}{7}{4}{$\bar{1}$}
	
	\boxWithLabel{0.5}{8}{4}{1}
	\boxWithLabel{0.5}{8.5}{4}{1}
	\boxWithLabel{0.5}{9}{4}{1}
	
	\boxWithLabel{0.5}{10}{4}{1}
	\boxWithLabel{0.5}{10.5}{4}{1}
	\boxWithLabel{0.5}{11}{4}{$\bar{1}$}

	\boxWithLabel{0.5}{2}{3}{2}
	\boxWithLabel{0.5}{2}{2.5}{1}
	
	\boxWithLabel{0.5}{3.5}{3}{$\bar{2}$}
	\boxWithLabel{0.5}{3.5}{2.5}{1}
	
	\boxWithLabel{0.5}{5}{3}{2}
	\boxWithLabel{0.5}{5}{2.5}{$\bar{1}$}
	
	\boxWithLabel{0.5}{6.5}{3}{$\bar{2}$}
	\boxWithLabel{0.5}{6.5}{2.5}{$\bar{1}$}

	\boxWithLabel{0.5}{8}{3}{1}
	\boxWithLabel{0.5}{8.5}{3}{1}
	\boxWithLabel{0.5}{8}{2.5}{$\bar{1}$}
	
	\boxWithLabel{0.5}{10}{3}{1}
	\boxWithLabel{0.5}{10.5}{3}{$\bar{1}$}
	\boxWithLabel{0.5}{10}{2.5}{$\bar{1}$}

	\boxWithLabel{0.5}{8}{1.5}{1}
	\boxWithLabel{0.5}{8}{1}{$\bar{1}$}
	\boxWithLabel{0.5}{8}{0.5}{$\bar{1}$}

	\boxWithLabel{0.5}{10}{1.5}{$\bar{1}$}
	\boxWithLabel{0.5}{10}{1}{$\bar{1}$}
	\boxWithLabel{0.5}{10}{0.5}{$\bar{1}$}
	
	\end{tikzpicture}
}
$$

Corteel, Savelief and Vuleti\'c \cite{corteelplane} use various methods to obtain the following generating function for plane overpartitions,
\begin{equation}\label{GenPl}
\overline{PL}(q):=\sum_{n=0}^{\infty}{\overline{pl}(n)q^n}=\prod_{n=1}^{\infty} \frac{(1+q^n)^{n}}{(1-q^n)^{n}}.
\end{equation}

Using the notation of Lovejoy and Mallet \cite{lovejoyncolor}, the generating function of plane overpartitions is also known as the generating function of $n$-color overpartitions. An $n$-color partition is a partition in which each number $n$ may appear in $n$ colors, with parts ordered first according to size and then according to color\footnote{We note that this is a different definition from what is often meant by $n$-color partition, in which each part regardless of the size may appear in one of $n$ colors}. For example, there are  6 $n$-color partitions of 3,
\begin{align*}
3_{3}, 3_{2},3_{1},2_{2}+1_{1},2_{1}+ 1_{1},1_{1}+1_{1}+1_{1}.
\end{align*}
 An $n$-color overpartition is defined similarly to be an $n$-color partition in which the final occurrence of a part $n_{j}$ may be overlined. For example, there are 16 $n$-color overpartitions of 3,
$$3_{3}, 3_{2}, 3_{1}, \overline{3}_{3}, \overline{3}_{2}, \overline{3}_{1},  2_{2}+1_{1}, \overline{2}_{2}+1_{1}, {2}_{2}+\overline{1}_{1}, \overline{2}_{2}+\overline{1}_{1}, {2}_{1}+{1}_{1},\overline{2}_{1}+{1}_{1}, {2}_{1}+\overline{1}_{1},\overline{2}_{1}+\overline{1}_{1},{1}_{1}+{1}_{1}+{1}_{1},{1}_{1}+{1}_{1}+\overline{1}_{1}.$$

In \cite{ali1}, the author defines a restricted form of plane overpartitions  called {\it{$k$-rowed plane overpartitions}} as plane overpartitions with at most $k$ rows. The total number of $k$-rowed plane overpartitions of $n$ is denoted by $\overline{pl}_{k}(n)$ and we define $\overline{pl}_{k}(0):=1$. The generating function is given by the following lemma.

 \begin{lemma}[Al-Saedi,\cite{ali1}]\label{lemma3}
 For a fixed positive integer $k$, the generating function for $k$-rowed plane overpartitions is given by
 \begin{equation}\label{PLOvkGen}
 \overline{PL}_{k}(q):=\sum_{n=0}^{\infty}{\overline{pl}_{k}(n)q^{n}}
 =\prod_{n=1}^{\infty}{\frac{(1+q^{n})^{\text{min}\{k,n\}}}{(1-q^{n})^{\text{min}\{k,n\}}}}.
 \end{equation}
 \end{lemma}
  
The author proves in \cite{ali1} that for all $n\geq 0$,
\begin{align*}
\overline{pl}_{4}(4n+1)+\overline{pl}_{4}(4n+2)+\overline{pl}_{4}(4n+3)\equiv 0\pmod{4}.
\end{align*}

\subsection{Main Results}\label{MainResultSec}

In this section, we state the main results of this paper. First, we start with results that involve plane and restricted plane overpartition congruences modulo $4$ and $8$.  Then, we state a few results for overpartition congruences modulo $8$ and congruence relations modulo $8$ between overpartitions and plane overpartitions.

 Recall that $\overline{p}_{o}(n)$ denotes the number of overpartitions of a positive integer $n$ into odd parts and $\overline{p}_{o}(0)=1.$
 
\begin{theorem}\label{PlOverPaTh}
For every integer $n\geq 1,$

\[   
	\overline{pl}(n)\equiv \overline{p}_{o}(n)\equiv 
	\begin{cases}
	2\pmod{4} &\text{if $n$ is a square or twice a square},\\
	0\pmod{4} &\text{otherwise}. \ 
	\end{cases}
	\]

\end{theorem}

For an integer $n$ and a prime $p$, let $ord_{p}(n)$ denote the highest nonnegative power of $p$ such that $p^{ord_{p}(n)}|n.$ The following theorem gives a congruence relation modulo $4$ between $\overline{pl}(n)$ and $ord_{p}(n)$ for each odd prime $p|n.$
\begin{theorem}\label{OvPa}
For any integer $n>1,$
\begin{equation}\label{opcon1}
\overline{pl}(n)\equiv 2\cdot \prod_{odd\; prime\; p|n}{\left(ord_{p}(n)+1\right)}\pmod{4}.
\end{equation}
\end{theorem}


 Next Theorem gives a systematic pattern of congruences modulo $4$ for even rowed plane overpartitions.

\begin{theorem}\label{th11}
	Let $k\geq 2$ be a positive even integer, $S_{k}:=\{j\;|\; j\;\mbox{odd}\;, 1\leq j \leq k-1\}$ and $\ell$ be the least common multiple of the integers in $S_{k}$. Then for any odd prime $p<k$, $1\leq r\leq ord_{p}(\ell),$ and  $n\geq 1,$
	\[   
	\overline{pl}_{k}(\ell n+p^r)\equiv
	\begin{cases}
	0\pmod{4} &\text{if $r$ is odd},\\
	2\pmod{4} &\text{if $r$ is even}. \ 
	\end{cases}
	\]
	Moreover, 
	for all $n\geq 1,$
	\[   
	\overline{pl}_{k}(\ell n)\equiv 
	\begin{cases}
	0\pmod{4} &\text{if $k\equiv 0\pmod{4}$},\\
	2\pmod{4} &\text{if $k\equiv 2\pmod{4}$}. \ 
	\end{cases}
	\]
\end{theorem}

In addition, we prove the following theorem which gives an equivalence modulo $4$ between the $k$-rowed plane overpartition function for odd integers $k$  and the overpartition function.

\begin{theorem}\label{th22}
	Let $k$ be a nonnegative integer. Then, for all $n\geq 0,$
	\begin{align}
	\overline{pl}_{2k+1}(2n+1)\equiv \overline{p}(2n+1)\pmod{4}\label{id1}.
	\end{align}
\end{theorem}

Next result gives a pattern of congruences modulo $4$ between $\overline{pl}_{k}(n)$ and $\overline{p}(n)$ for odd $k$.

\begin{theorem}\label{th333}
Let $k\geq 2$ and $\ell$ be the least common multiple of all positive even integers $\leq 2k$. Then for all integers $n\geq 1,$ 
\begin{align}\label{th333eq1}
\overline{pl}_{2k+1}(\ell n+2^j)\equiv \overline{p}(\ell n+2^j)\pmod{4},
\end{align} 
where $j\geq 2, j\equiv 0\pmod{2}$ and $2^{j-1}\leq k.$ Moreover, if $k\equiv 0\pmod{2}$, then for all integers $n\geq 0$ 
\begin{align}\label{th333eq2}
\overline{pl}_{2k+1}(\ell n)\equiv \overline{p}(\ell n)\pmod{4}.
\end{align} 
\end{theorem} 

 Next theorem gives a few examples of $4$ and $8$-rowed plane overpartition congruences modulo $8$. One may find more of this type using similar methods of proof. 
  
\begin{theorem}\label{ThmMod8}
	For all integer $n\geq 1,$
	\begin{align}\label{c11}
	\overline{pl}_{4}(12n)\equiv 0\pmod{8},
	\end{align}
	\begin{align}\label{c22}
	\overline{pl}_{4}(6n+3)\equiv 0\pmod{8},
	\end{align}
	\begin{align}\label{c33}
	\overline{pl}_{8}(210n)\equiv 0\pmod{8},
	\end{align}
	\begin{align}\label{c34}
	\overline{pl}_{8}(210n+3)\equiv 0\pmod{8},
	\end{align}
	\begin{align}\label{c35}
	\overline{pl}_{8}(210n+9)\equiv 0\pmod{8},
	\end{align}
	\begin{align}\label{c37}
	\overline{pl}_{8}(210n+105)\equiv 0\pmod{8}.
	\end{align}
\end{theorem}

Next result gives a useful overpartition congruence modulo $8.$

\begin{theorem}\label{CongOverMod8}
The following holds for all nonsquare odd integers $n\geq 0,$
\begin{align*}
\overline{p}(n)\equiv 0\pmod{8}.
\end{align*}
\end{theorem}

For $k$-rowed plane overpartitions with odd $k=5$, we obtain the following equivalence modulo $8$ for plane overpartitions with at most $5$ rows.

	\begin{theorem}\label{5rowed}
	The following holds for all $n\geq 0,$
	\begin{align}\label{Con1Pl5}
	\overline{pl}_{5}(12n+1)\equiv 	\overline{p}(12n+1)\pmod{8},
	\end{align}
	\begin{align}\label{ConPl5}
	\overline{pl}_{5}(12n+5)\equiv 	\overline{p}(12n+5)\pmod{8}.
	\end{align}
	\end{theorem}

The rest of this paper will be organized as follows. In Section \ref{pre}, we review some preliminaries which are needed in the proofs of main the theorems including a useful theorem of Kwong \cite{kwong1} which we will apply to prove some of the identities in Theorem \ref{ThmMod8}. In Section \ref{MainProofs}, we present the proofs of the main results in this paper, and we give some applications for these results. In Section \ref{Remark},  we conclude with final remarks.  

\section{\textbf{Preliminaries}}\label{pre}

In this section, we shed light on the periodicity of a certain type of $q$-series, their minimum periodicity modulo integers and how to find such periodicity. Kwong and others have done extensive studies on the periodicity of certain rational functions, including partition generating functions, for example see  \cite{kwong3}, \cite{kwong1}, \cite{kwong2}, \cite{newman}, and \cite{nijenhuis1987}. We will apply a result of Kwong \cite{kwong1} that provides us a systematic formula to calculate the minimum periodicity modulo prime powers of such periodic series.

Let $$A(q)=\sum_{n=0}^{\infty}{\alpha(n)q^{n}} \in \mathbb{Z}[[q]]$$ be a formal power series with integer coefficients, and let $d, \ell$ and $\gamma$ be positive integers. We say $A(q)$ is {\it{periodic}} with period $d$ modulo $\ell$ if, for all $n\geq \gamma$,
$$\alpha(n+d)\equiv \alpha(n)\pmod{\ell}.$$
The smallest such period for $A(q)$, denoted $\pi_{\ell}(A)$, is called the {\it{minimum period of}} $A(q)$ modulo $\ell$. $A(q)$ is called {\it{purely periodic}} if $\gamma=0$. In this work, periodic always means purely periodic.

For example, consider the $q$-series $A(q)=\sum_{n\geq 0}{\alpha(n)q^n}$ which generates the sequence $\alpha(n):=4n+1$ for all $n\geq 0$. Note that $\alpha(n+2k)-\alpha(n)=8k\equiv 0\pmod{8}$ for all $n\geq 0$ and $k\geq 1$. Thus, $A(q)$ is periodic modulo $8$ and for each $k$, there is a period of length $2k$. Thus, the minimum period modulo $8$ is  $\pi_{8}(A)=2$.

 Before we state a result of Kwong \cite{kwong1}, we state some necessary definitions.\\
 For an integer $n$ and prime $\ell$, define $ord_{\ell}(n)$ to be the unique nonnegative integer such that $$\ell^{ord_{\ell}(n)}\cdot m=n,$$ where $m$ is an integer and $\ell\nmid m$. In addition, we call $m$ the {\it{$\ell$-free part}} of $n$.
 
  For a finite multiset of positive integers $S$, we define $m_{\ell}(S)$ to be the $\ell$-free part of $lcm\{n|n\in S\}$, and $b_{\ell}(S)$ to be the least nonnegative integer such that 
 $$\ell^{b_{\ell}(S)}\geq \sum_{n\in S}\ell^{ord_{\ell}(n)}.$$

We now state Kwong's theorem.

\begin{theorem}[Kwong,\cite{kwong1}]\label{kwong}
	Fix a prime $\ell$, and a finite multiset $S$ of positive integers. Then for any positive integer $N$, $$A(q)=\sum_{n=0}^{\infty} {p(n;S)q^{n}}$$ is periodic modulo $\ell^{N}$, with minimum period $$\pi_{\ell^{N}}(A)=\ell^{N+b_{\ell}(S)-1}\cdot m_{\ell}(S).$$
\end{theorem}

For example, let $S=\{1_{1},1_{2}, 2_{1},2_{2},2_{3},4_{1},4_{2},5\}.$ Then $p(n;S)$ is generated by the following $q$-series
\begin{align*}
A(q):=\sum_{n=0}^{\infty}{p(n;S)q^n}=\prod_{n\in S}{\frac{1}{(1-q^n)}}=\frac{1}{(1-q)^2(1-q^2)^3(1-q^4)^2(1-q^5)}.
\end{align*}
Letting $\ell=2$ in Theorem \ref{kwong}, we obtain 
\begin{align*}
2^{b_{2}(S)}\geq \sum_{n\in S} 2^{ord_{2}(n)}=2\cdot 2^{0}+3\cdot 2^{1}+2\cdot 2^2+ 2^0=17.
\end{align*}
Thus $b_{2}(S)=5, \;lcm \{n: n\in S\}= 20$, and hence $m_{2}(S)= 5.$ Using Theorem \ref{kwong}, for a positive integer $N$, the minimum period of $A(q)$ modulo $2^N$ is $\pi_{2^N}(A)=2^{N+4}\cdot 5.$

Theorem \ref{kwong} was used by the author in \cite{ali1} to prove a method to obtain various partition theoretic congruences by verifying they hold for a finite number of values. This work generalized a result of Mizuhara, Sellers, and Swisher \cite{periodic}.  
 
The following lemma has a flavor of periodicity of restricted partitions. It is an application of Theorem \ref{kwong} and will be used in the proof of Theorem \ref{ThmMod8}

\begin{lemma}\label{LemPart}
	Let  $a,b,c\geq 2$ be integers such that $a,b$ and $c$ are pairwise relatively prime. Let $M_{c}$ be the number of pairs of positive integers $(n,m)\in \mathbb{N}^2$ with $an+bm=c$ where $M_{c}:=0$ if no such pairs exists. Then,
	\begin{align*}
     M_{c}=p(c;\{a,b\}),
	\end{align*}
	where $p(c;\{a,b\})$ is the number of partitions of $c$ into parts from the set $\{a,b\}$. Moreover, for every integer $N\geq 1$ and a prime $\ell$,
	\begin{align*}
    M_{c+\pi_{\ell^N}}\equiv M_{c}\pmod{\ell^N},
   	\end{align*}
   	where $\pi_{\ell^N}$ is the minimum period modulo $\ell^N$ of the $q$-series 
   	\begin{align}\label{qser}
\sum_{n=0}^{\infty}p(n;\{a,b\})q^n=\frac{1}{(1-q^a)(1-q^b)}
   	\end{align}
   	 which  generates the partitions with parts from $\{a,b\}$.
\end{lemma}
\begin{proof}
	Note that if there are two positive integers $n$ and $m$ such that $an+bm=c$, then $c$ can be partitioned into parts form $\{a,b\}$ as follows 
	\begin{align*}
    \underbrace{a+\dots+a}_{n\text{-times}}\vphantom{1}+\underbrace{b+\dots+b}_{m\text{-times}}\vphantom{1}=c.
	\end{align*} 
	Thus, any pair of positive integers $n$ and $m$ that satisfy $an+bm=c$ corresponds to a partition of $c$ into parts from $\{a,b\}$. Likewise, since $gcd(a,b)=gcd(a,c)=gcd(b,c)=1,$ then any such partition of $c$ must involve both $a$ and $b$, and hence any corresponding integers $n$ and $m$ must be positive. By considering all such pairs $(n,m)$, we then obtain
	\begin{align*}
	  M_{c}=p(c;\{a,b\}).
	\end{align*}
	By Theorem \ref{kwong}, the $q$-series \eqref{qser} is periodic modulo $\ell^N$ for any integer $N\geq 1$ and a prime $\ell$, with minimum period $\pi_{\ell^N}=\ell^{N+b_{\ell}(\{a,b\})-1}\cdot m_{\ell}(\{a,b\})$ which yields that
	\begin{align*}
    M_{c+\pi_{\ell^N}}=p(c+\pi_{\ell^N};\{a,b\})\equiv p(c;\{a,b\})= M_{c}\pmod{\ell^N}.
	\end{align*}
\end{proof}

Also, the following lemma will be used in the proof of Theorem \ref{ThmMod8}.
 
 \begin{lemma}\label{lemma(ab)}
 	Let $a,b,c \in \mathbb{N}$ such that $gcd(a,b)=1$. Then there are $c-1$ pairs of positive integers $(n,m)$ such that $an+bm=abc$.
 \end{lemma}
 \begin{proof}
 	Suppose that $an+bm=abc$. Then $an=abc-bm$ and so $b|an$, and since $gcd(a,b)=1$, we must have $b|n$. So $n=bN$ for some $N\in \mathbb{N}$. Similarly, $a|m$ and so $m=aM$ for some $M\in \mathbb{N}$. We see then that $abN+abM=abc$ and thus $N+M=c$. Hence, if $(n,m)\in \mathbb{N}^2$ satisfies $an+bm=abc$, then it is equivalent to say there is a pair $(N,M)\in \mathbb{N}^2 $ such that $N+M=c$. Note that there are $c-1$ pairs $(N,M)\in \mathbb{N}^2$ such that $N+M=c$ since the possible ways are $1+(c-1), 2+(c-2),\dots,(c-1)+1$.
 \end{proof}
 
We define throughout the formal power series 
 \begin{align}\label{eqf}
 f(q):=\frac{1+q}{1-q}=1+2q+2q^2+2q^3+\cdots.
 \end{align}
 Note that for every positive integer $n\geq 1$,
 \begin{align*}
 f(q^n)\equiv \frac{1-q^n}{1-q^n} \equiv 1\pmod{2}.
 \end{align*} 
 Thus, we obtain
 \begin{align*}
 \sum_{n=0}^{\infty}{\overline{pl}(n)q^n}=\prod_{n=1}^{\infty} \frac{(1+q^n)^{n}}{(1-q^n)^{n}}
  =\prod_{n=1}^{\infty}{f(q^n)^n}\equiv 1\pmod{2},
 \end{align*}
 and
  \begin{equation}
 \sum_{n=0}^{\infty}{\overline{pl}_{k}(n)q^{n}}
  =\prod_{n=1}^{\infty}{\frac{(1+q^{n})^{\text{min}\{k,n\}}}{(1-q^{n})^{\text{min}\{k,n\}}}}=\prod_{n=1}^{\infty}f(q^n)^{\text{min}\{k,n\}}\equiv 1\pmod{2}.
  \end{equation}

 \begin{lemma}\label{MainLemma}
 	For all $k\geq 1$,
 	\begin{align}\label{LemEq}
 	\left(1+2S(q)\right)^{2^k}&\equiv 1\pmod{2^{k+1}},
 	\end{align}	
 where $S(q)\in \mathbb{Z}[[q]]$ is a $q$-series with integer coefficients.
 \end{lemma}
 \begin{proof}
 We induct on $k$. It is easy to see that \eqref{LemEq} is true for $k=1.$ Now suppose that \eqref{LemEq} is true for $1\leq j\leq k-1$. Then by induction there is a $q$-series $T(q)\in \mathbb{Z}[[q]]$ such that $\left(1+2S(q)\right)^{2^{k-1}}=1+2^{k}T(q).$ Thus,
 \begin{align*}
 \left(1+2S(q)\right)^{2^k}=&\left((1+2S(q))^{2^{k-1}}\right)^{2}\\
 &=\left(1+2^kT(q)\right)^{2}\\
 &\equiv 1\pmod{2^{k+1}},
 \end{align*}
 as desired.
 \end{proof}
 
 The following lemma is a very useful tool in the proofs of the main results.
 \begin{lemma}\label{lem1}
 	For all integers $n,k\geq 1$,
 	\begin{align*}
 	f(q^n)^{2^{k}}\equiv 1\pmod{2^{k+1}}.
 	\end{align*}	
 \end{lemma}
 \begin{proof}	Let $S(q):=\sum_{m\geq 1}q^m$. We observe that $S(q)=\frac{q}{1-q}$, and so 
 	\begin{align}\label{eqf2}
 	f(q^n)=\frac{1+q^n}{1-q^n}=1+\frac{2q^n}{1-q^n}=1+2S(q^n).
 	\end{align}
 	The conclusion then follows by Lemma \ref{MainLemma}.
 \end{proof}
 
 Overpartition congruences modulo small powers of $2$ can be derived from the following fact proved by Hirschhorn and Sellers  [\cite{hirschhorn2006arithmetic}, Theorem $2.1$] which states
 \begin{align}\label{PhiThm}
 \overline{P}(q)=\phi(q)\overline{P}(q^2)^2.
 \end{align}
 Iterating \eqref{PhiThm} yields that [\cite{hirschhorn2006arithmetic}, Theorem $2.2$]
 \begin{align*}
 \overline{P}(q)=\phi(q)\;\phi^2(q^2)\;\phi^4(q^4)\;\phi^8(q^8)\cdots.
 \end{align*}
 Thus, 
 \begin{align*}
 \sum_{n=0}^{\infty}\overline{p}(n)q^{n}=\left(1+2\sum_{n\geq 1}{q^{n^2}}\right) \left(1+2\sum_{n\geq 1}{q^{2n^2}}\right)^2 \left(1+2\sum_{n\geq 1}{q^{4n^2}}\right)^4 \left(1+2\sum_{n\geq 1}{q^{8n^2}}\right)^8\cdots
 \end{align*}
 By Lemma \ref{MainLemma}, we observe that for all $k\geq 1,$
 \begin{align}\label{GenParTwo}
 \phi(q)^{2^k}\equiv 1\pmod{2^{k+1}}.
 \end{align}
 Thus, by \eqref{GenParOne} and \eqref{GenParTwo}, we obtain the following general equivalence modulo $2^{k},$ for $k\geq 2,$
 \begin{align}\label{GenConOver}
 \sum_{n=0}^{\infty}\overline{p}(n)q^{n}\equiv \prod_{j=0}^{k-2}\left(\phi(q^{2^j})\right)^{2^j}\equiv 1+\sum_{j=1}^{k-1}{2^j}\sum_{n=1}^{\infty}{(-1)^{n+j}c_{j}(n)q^n}\pmod{2^{k}},
 \end{align}
 which for the case $k=2$, we obtain
 \begin{align}\label{OverPartMod2}
 \sum_{n=0}^{\infty}\overline{p}(n)q^{n}\equiv\phi(q)\equiv 1+2\sum_{n=1}^{\infty}{q^{n^2}}\pmod{4},
 \end{align}
 which yields for each nonsquare integer $n\geq 1,$
 \begin{align}\label{NonSqMod4}
 \overline{p}(n)\equiv 0\pmod{4}.
 \end{align}
 
 Manipulating the generating function $\overline{P}(q)$ of overpartitions, Hirschhorn and Sellers \cite{hirschhorn2005arithmetic} employed elementary dissection techniques of generating functions and derived a set of overpartition congruences modulo small powers of $2$. For example, they prove that for all $n\geq 0,$
 \begin{align*}
 &\overline{p}(9n + 6)\equiv 0\pmod{8},\\
 &\overline{p}(8n+7)\equiv 0\pmod{64}.
 \end{align*}
 
 For a modulus that is not a power of $2$, Hirschhorn and Sellers \cite{hirschhorn} prove the first infinite family of congruences for $\overline{p}(n)$ modulo $12$ by showing first that for all $n\geq 0,$ and all $\alpha \geq 0,$
 \begin{align*}
 \overline{p}(9^\alpha(27n+18))\equiv 0\pmod{3}.
 \end{align*}
 Together with the fact $9^\alpha(27n+18)$  is nonsquare for all $n\geq 0, \alpha \geq 0,$ and hence by the help of \eqref{NonSqMod4}, it follows that for all $\alpha, n\geq 0$,
 \begin{align}\label{seller1}
 \overline{p}(9^\alpha(27n+18))\equiv 0\pmod{12}.
 \end{align}
 
 Several examples of overpartition congruences have been found. For more examples of overpartition congruences, one may refer to work of Chen and Xia \cite{chen2013proof}, Fortin, Jacob and Mathieu \cite{fortin2005jagged}, Treneer \cite{treneer2006congruences} and Wang \cite{wang2014another}.
 
 Now, let $\overline{p}_{o}(n)$ denote the number of overpartitions of $n$ into odd parts. The generating function for $\overline{p}_{o}(n)$ \cite{hirschhorn2006arithmetic} is given by 
 \begin{align}\label{OverOddGen}
 \overline{P}_{o}(q):=\sum_{n=0}^{\infty}{\overline{p}_{o}(n)q^n}=\prod_{n=1}^{\infty}\frac{1+q^{2n-1}}{1-q^{2n-1}}.
 \end{align}
 
 Similar to \eqref{PhiThm}, The generating function $\overline{P}_{o}(q)$ can be written as [see \cite{hirschhorn2006arithmetic},Theorem $2.3$], 
 \begin{align}\label{PhiOdd}
 \overline{P}_{o}(q)=\phi(q)\overline{P}(q^2),
 \end{align}
 and the iteration of \eqref{PhiOdd} yields [\cite{hirschhorn2006arithmetic},Theorem $2.4$],
 \begin{align}
 \overline{P}_{o}(q)=\phi(q) \phi(q^2) \phi^2(q^4) \phi^4(q^8)\cdots
 \end{align}
 For modulus $4$, we then easily get
 \begin{align*}
 \sum_{n=0}^{\infty}\overline{p}_{o}(n)q^{n}\equiv \phi(q)\phi(q^2)\equiv 1+2\sum_{n\geq 1}{q^{n^2}}+2\sum_{n\geq 1}{q^{2n^2}}\pmod{4}. 
 \end{align*}
 As a consequence, Hirschhorn and Sellers obtain Theorem $2.3$ of \cite{hirschhorn2006arithmetic} as following.
 \begin{theorem}[Hirschhorn, Sellers, \cite{hirschhorn2006arithmetic}]\label{OverOddGenTh}
 For every integer $n\geq 1,$
 \[   
 	\overline{p}_{o}(n)\equiv
 	\begin{cases}
 	2\pmod{4} &\text{if $n$ is a square or twice a square},\\
 	0\pmod{4} &\text{otherwise}. \ 
 	\end{cases}
 	\]
 \end{theorem}
 
  Similar to \eqref{GenConOver}, we have the following general equivalence modulo $2^{k}$ for all $k\geq 2,$
 \begin{align*}
 \sum_{n=0}^{\infty}\overline{p}_{o}(n)q^{n}\equiv \phi(q)\; \phi(q^2)\; \phi^2(q^4)\cdots \phi(q^{2^{k-1}})^{2^{k-2}}\pmod{2^{k}}.
 \end{align*}
 
  Later, we will revisit the equivalences \eqref{GenConOver}, \eqref{NonSqMod4}, and Theorem \ref{OverOddGenTh}.

\section{\textbf{Proofs of Main Results and Some Corollaries}}\label{MainProofs}
We now present proofs of our main results stated in Section \ref{MainResultSec}. In addition, we give several corollaries.

\begin{proof}[\textbf{Proof of Theorem \ref{PlOverPaTh}}]
We observe by \eqref{GenPl} and Lemma \ref{lem1}, the generating function for plane overpartitions
\begin{align*}
\sum_{n=0}^{\infty}{\overline{pl}(n)q^n}=&\prod_{n=1}^{\infty} \frac{(1+q^n)^{n}}{(1-q^n)^{n}}=\prod_{n=1}^{\infty}f(q^n)^n\\
 =&\prod_{n=1}^{\infty}f(q^{2n})^{2n}f(q^{2n-1})^{2n-1}\\
 \equiv& \prod_{n=1}^{\infty}f(q^{2n-1})\pmod{4}\\
 =&\prod_{n=1}^{\infty}\frac{(1+q^{2n-1})}{(1-q^{2n-1})}\pmod{4}\\
 =& \sum_{n=0}^{\infty}{\overline{p}_{o}(n)q^{n}}\pmod{4}.
\end{align*}
Thus for all $n\geq 1,$
\begin{align*}
\overline{pl}(n)\equiv \overline{p}_{o}(n)\pmod{4}.
\end{align*}
By Theorem \ref{OverOddGenTh}, for all $n\geq 1,$
\[   
	\overline{p}_{o}(n)\equiv
	\begin{cases}
	2\pmod{4} &\text{if $n$ is a square or  twice a square},\\
	0\pmod{4} &\text{otherwise}, \ 
	\end{cases}
	\]
and the result follows.
\end{proof}

\begin{corollary}
The following holds for all $n\geq 0,$
\begin{align*}
\overline{pl}(4n+3)\equiv 0\pmod{4}.
\end{align*}
\end{corollary}
\begin{proof}
Note that for all $n\geq 0$, $4n+3$ is not a square since positive odd squares are $1$ modulo $4$. Also $4n+3$ is odd so it can not be twice a square for all $n\geq 0$. The result then follows by Theorem \ref{PlOverPaTh}.
\end{proof}

\begin{proof}[\textbf{Proof of Theorem \ref{OvPa}}]
Following the same procedure in Theorem \ref{PlOverPaTh}, we note that
\begin{alignat}{3}\label{divisor}
\sum_{n=0}^{\infty}{\overline{pl}(n)q^{n}} 
& \equiv f(q) \cdot f(q^3)\cdot f(q^{5})\cdots \pmod{4}\nonumber\\ 
&\equiv \left( 1+2\sum_{m\geq 1}{q^m}\right)\cdot \left( 1+2\sum_{m\geq 1}{q^{3m}}\right)\cdot \left( 1+2\sum_{n\geq 1}{q^{5m}}\right)\cdots  \pmod{4} \nonumber \\
&\equiv 1+2\sum_{m\geq 1}{q^m}+2\sum_{m\geq 1}{q^{3m}}+2\sum_{m\geq 1}{q^{5m}}+\cdots \pmod{4}\nonumber\\
&\equiv 1+2\sum_{m\geq 1}{(q^m+q^{3m}+q^{5m}+\cdots)} \pmod{4}. 
\end{alignat}

Now for any integer $n>1,$ by the fundamental theorem of arithmetic, $n$ can be written as a product of prime powers. Thus,
\begin{align}\label{factor}
n=2^{\alpha_{0}}p_{1}^{\alpha_{1}}\cdots p_{k}^{\alpha_{k}},
\end{align}
 where $p_{i}$ are primes and $\alpha_{0},\alpha_{i}$ are nonnegative integers for each $i=1,\dots,k$. Thus $ord_{p_{i}}(n)=\alpha_{i}$ for each $i=1,\dots,k.$  Note that the term $q^n$ will occur in the series
 \begin{align*}
 \sum_{m\geq 1}{(q^m+q^{3m}+q^{5m}+\cdots)}
 \end{align*}
   when $m=n/d$ in $q^{dm}$ where $d$ is an odd divisor of $n$.  In terms of the prime factorization of $n$ in \eqref{factor}, the number of odd divisors of $n$ is given by
 $$\prod_{i=1}^{k}{\left(\alpha_{i}+1 \right)}=\prod_{i=1}^{k}{\left(ord_{p_{i}}(n)+1\right)}=\prod_{odd\; prime\; p|n}{\left(ord_{p}(n)+1\right)}.$$
  Thus the coefficient of $q^n$ in \eqref{divisor} is then given by
 $$2\cdot \prod_{odd\; prime\; p|n}{\left(ord_{p}(n)+1\right)}.$$
  \end{proof}
We now prove Theorem \ref{th11}.

\begin{proof}[\textbf{Proof of Theorem \ref{th11}}]
Let $k\geq 2$ be even. We first observe that the generating function of the $k$-rowed plane overpartitions can be rewritten modulo $4$ using \eqref{PLOvkGen} and Lemma \ref{lem1} to obtain
	\begin{alignat*}{3}
	 \sum_{n=0}^{\infty}{\overline{pl}_{k}(n)q^{n}} &=  \prod_{n=1}^{\infty}{\frac{(1+q^{n})^{\text{min}\{k,n\}}}{(1-q^{n})^{\text{min}\{k,n\}}}} =\prod_{n=1}^{\infty}f(q^n)^{\text{min}\{k,n\}}\\
	&= f(q) f(q^2)^2\cdots f(q^{k-1})^{k-1}\cdot \prod_{n\geq k}{f(q^n)^k} \\
	& \equiv f(q) f(q^3)\cdots f(q^{k-1})\pmod{4}\\ 
	&\equiv \left( 1+2\sum_{n\geq 1}{q^n}\right)\cdot \left( 1+2\sum_{n\geq 1}{q^{3n}}\right)\cdots \left( 1+2\sum_{n\geq 1}{q^{(k-1)n}}\right) \pmod{4}\\
	&\equiv 1+2\sum_{n\geq 1}{q^n}+2\sum_{n\geq 1}{q^{3n}}+\cdots +2\sum_{n\geq 1}{q^{(k-1)n}}\pmod{4}.
	\end{alignat*}
Thus, we see that	
	\begin{alignat*}{3}
	\sum_{n=0}^{\infty}{\overline{pl}_{k}(n)q^{n}}& \equiv 1+2\sum_{i\in S_{k}}\sum_{n\geq 1}{q^{in}}\pmod{4}\\
	& \equiv 1+2\sum_{i\in S_{k}}\sum_{in\not\equiv 0(\text{mod} \;\ell )}{q^{in}}+2\sum_{i\in S_{k}}\sum_{n\geq 1}{q^{\ell n}}\pmod{4}\\
	& \equiv 1+2\sum_{i\in S_{k}}\sum_{in\not\equiv 0(\text{mod} \;\ell )}{q^{in}}+2|S_{k}|\sum_{n\geq 1}{q^{\ell n}}\pmod{4}\\
	& \equiv 1+2\sum_{i\in S_{k}}\sum_{in\not\equiv 0(\text{mod} \;\ell )}{q^{in}}+k\sum_{n\geq 1}{q^{\ell n}} \pmod{4},
	\end{alignat*}
where the last congruence is obtained using the fact that $|S_{k}|=k/2$. Thus, we obtain that
	\[   
	\overline{pl}_{k}(\ell n)\equiv 
	\begin{cases}
	0\pmod{4} &\text{if $k\equiv 0\pmod{4}$}\\
	2\pmod{4} &\text{if $k\equiv 2\pmod{4}$.} \ 
	\end{cases}
	\] 
Now for a prime $p \in S_{k}$ and $s:=ord_{p}(\ell)$, we let
	\begin{align*}
	\sum_{n\geq 1}{\alpha(n)q^n}:=\sum_{n\geq 1}{\left(q^{n}+q^{pn}+q^{p^{2}n}+\dots +q^{p^{s}n}\right)}.
	\end{align*}
	For any $m\geq 1$ and $1\leq r\leq s$, the term $q^{m\ell+p^r}$ will occur in the above series when $n=m\ell+p^r,\frac{m\ell+p^r}{p}, \dots, \frac{m\ell+p^r}{p^r},$ arising from the terms $q^{n}, q^{pn}, \dots, q^{p^{r}n},$ respectively. The term $q^{m\ell+p^r}$ can not be obtained from $\sum_{n\geq 1}{q^{p^in}}$ for $i>r$, since $p^i$ does not divide $m\ell+p^r$. Thus
	
	\[   
	\alpha(m\ell+p^r) =r+1\equiv
	\begin{cases}
	0\pmod{2} &\text{if $r$ is odd}\\
	1\pmod{2} &\text{if $r$ is even}.\ 
	\end{cases}
	\] 

Observe that
	\begin{align*}
	\sum_{j\in S_{k}}\sum_{n\geq 1}{q^{jn}}
	&=\sum_{n\geq 1}{\left(q^n+q^{pn}+\dots+q^{p^{s}n}\right)}+\sum_{j\in S_{k}-\{p^{i}:0\leq i\leq s\}}\sum_{n\geq 1}{q^{jn}}\\
	&=\sum_{n\geq 1}{\alpha(n)q^n}+\sum_{j\in S_{k}-\{p^{i}:0\leq i\leq s\}}\sum_{n\geq 1}{q^{jn}}.
	\end{align*}
Thus, we obtain
\begin{align}\label{plk}
\sum_{n=0}^{\infty}{\overline{pl}_{k}(n)q^{n}}\equiv 1+2\sum_{n\geq 1}{\alpha(n)q^n}+2\sum_{j\in S_{k}-\{p^{i}:0\leq i\leq s\}}\sum_{n\geq 1}{q^{jn}} \pmod{4}.
\end{align}
	Also, we note that for all $n,m\geq 1$, if $j\in S_{k}-\{p^{i}:0\leq i\leq s\},$ then $jn \not = \ell m+p^{r}$ for all $1\leq r\leq s$. If not, then there are two positive integers $n_{0}, m_{0}$ such that $jn_{0} = \ell m_{0}+p^{r}$, and thus $n_{0}= (\ell m_{0}+p^{r})/j$. Since by the choice of $\ell$, we have that $j$ divides $\ell$, then $j$ must divide $p^{r}$ which contradicts that $j\not =p^{i}$ for all $0\leq i\leq s$.
	Thus terms of the form $q^{\ell n+p^{r}}$ will arise in $\sum_{j\in S_{k}}\sum_{n\geq 1} q^{jn}$ only from $\sum_{n\geq 1} \alpha(n) q^n$.
	 
	Now, If we extract the terms of the form $q^{\ell n+p^{r}} $ and replace $n$ with $\ell n+p^{r}$ in \eqref{plk}, we find that,
	\begin{align*}
	\sum_{n\geq 1}{\overline{pl}_{k}(\ell n+p^{r})q^{n}}\equiv 2\cdot \sum_{n\geq 1}{\alpha(\ell n+p^{r}) q^{n}} \equiv 2(r+1)\sum_{n\geq 1}{q^n} \pmod{4}.
	\end{align*}
	
 Thus, modulo $4$,
	\[   
	\ \overline{pl}_{k}(\ell n+p^{r})\equiv 2\alpha(\ell n+p^r) =2(r+1)\equiv
	\begin{cases}
	0\pmod{4} &\text{if $r$ is odd,}\\
	2\pmod{4} &\text{if $r$ is even}.\ 
	\end{cases}
	\] 
\end{proof}	

As an application of Theorem \ref{th11}, we give a few examples in the following corollary.
 \begin{corollary}
 	The following hold for all $n\geq 1$,
 	\begin{align*}
\overline{pl}_{4}(3n)\equiv 0\pmod{4},
 	\end{align*}
 	\begin{align*}
\overline{pl}_{6}(15n+b)\equiv 0\pmod{4}, \;\;	for\;\; b\in \{3,5\},
 	\end{align*}
 \begin{align*}
\overline{pl}_{6}(15n)\equiv 2\pmod{4},
 \end{align*}
  \begin{align*}
\overline{pl}_{8}(105n+b)\equiv 0\pmod{4}, \;\;	for\;\; b\in \{0,3,5,7\},
 	\end{align*}
 	\begin{align*}
\overline{pl}_{10}(315n+b)\equiv 0\pmod{4},  \;\;	for\;\; b\in \{3,5,7\},
 	\end{align*}
 	\begin{align*}
\overline{pl}_{10}(315n+b)\equiv 2\pmod{4},  \;\;	for\;\; b\in \{0,9\},
 	\end{align*}
 	\begin{align*}
 	 	\overline{pl}_{12}(3465n+b)\equiv 0\pmod{4}, \;\;	for\;\; b\in \{0,3,5,7,11\},
 	 \end{align*}
\begin{align*}
 	\overline{pl}_{12}(3465n+9)\equiv 2\pmod{4}.
 \end{align*}
    
\end{corollary} 

\begin{proof}
For the first congruence, letting $k=4$, we have that $S_{4}=\{1,3\}$, and $\ell=3$. Since $k\equiv 0\pmod{4}$, by Theorem \ref{th11}, for all $n\geq 1$
\begin{equation*}
\overline{pl}_{4}(3n)\equiv 0 \pmod{4}.
\end{equation*}
Now to see the second and third congruences, let $k=6$, then $S_{6}=\{1,3,5\}$ and $\ell=15$. The only primes in $S_{6}$ are 3 and 5 with $ord_{p}(\ell)=1$ for $p=3,5$. Hence $r=1$ is the only choice for $1\leq r\leq ord_{p}(\ell)$. Thus by Theorem \ref{th11}, for all $n\geq 1,$
\begin{alignat*}{2}
&\overline{pl}_{6}(15n+3)&\equiv 0&\pmod{4},\\
&\overline{pl}_{6}(15n+5)&\equiv 0&\pmod{4}.
\end{alignat*}
Moreover, $k=6\equiv 2 \pmod{4}$ which yields that for all $n\geq 1,$
\begin{align*}
\overline{pl}_{6}(15n)\equiv 2\pmod{4}.
\end{align*}
The rest of the identities can be proved similarly.
\end{proof}

\begin{proof}[\textbf{Proof of Theorem \ref{th22}}]
  Clearly, for $k=0,\; \overline{pl}_{1}(n)= \overline{p}(n),$ for all $n\geq 0.$ Now, for $k\geq 1$, we first define 
  \begin{align*}
  g(q):=\frac{1}{f(q)},
  \end{align*}
  and note that by Lemma \ref{lem1} and \eqref{eqf2} that
 	\begin{align}\label{Theq}
 	g(q)\equiv f(q)=1+2\sum_{n\geq 1}{q^n}\pmod{4}.
 	\end{align}
 	and recall the generating function of $2k+1$-rowed plane overpartions \eqref{PLOvkGen},
 	\begin{alignat*}{3}
 	\sum_{n=0}^{\infty}{\overline{pl}_{2k+1}(n)q^{n}}
 	&=\prod_{n=1}^{\infty}{\frac{(1+q^{n})^{\text{min}\{2k+1,n\}}}{(1-q^{n})^{\text{min}\{2k+1,n\}}}}= \prod_{n=1}^{\infty}{f(q^n)^{\text{min}\{2k+1,n\}}}& \\
 	&= f(q) f(q^2)^2\cdots f(q^{2k})^{2k}\cdot \prod_{n\geq 2k+1}{f(q^n)^{2k+1}}& \\
 	&\equiv f(q) f(q^3)\cdots f(q^{2k-1})\cdot \prod_{n\geq 2k+1}{f(q^n)} \pmod{4},&
 	\end{alignat*}
 where the last congruence is by Lemma \ref{lem1}. Thus, we have by \eqref{Theq} that
 	\begin{align*}
 	\sum_{n=0}^{\infty}{\overline{pl}_{2k+1}(n)q^{n}}
     & \equiv g(q^2) g(q^4) g(q^6)\cdots g(q^{2k}) \cdot \prod_{n=1}^{\infty}{f(q^n)}\pmod{4} &\\
     & \equiv f(q^2) f(q^4) f(q^6)\cdots f(q^{2k}) \cdot \prod_{n=1}^{\infty}{f(q^n)}\pmod{4} &\\
 	& \equiv \left( 1+2\sum_{n\geq 1}{q^{2n}}\right)\cdot \left( 1+2\sum_{n\geq 1}{q^{4n}}\right)\cdot \left( 1+2\sum_{n\geq 1}{q^{6n}}\right)\cdots  &\\
 	&\;\;\;\;\left( 1+2\sum_{n\geq 1}{q^{2kn}}\right) \cdot 
 	\left(\sum_{n\geq 0}{\overline{p}(n)q^n}\right)\pmod{4}.&
 	\end{align*}
 	Note that for all $n\geq 1, \overline{p}(n)\equiv 0\pmod{2},$  and hence $2\overline{p}(n)\equiv 0\pmod{4}.$ Consequently,
 	\begin{alignat}{2}\label{eqth11}
 	\sum_{n=0}^{\infty}{\overline{pl}_{2k+1}(n)q^{n}}&\equiv
 	 \left(1+2 \sum_{n\geq 1}{\left(q^{2n}+q^{4n}+ q^{6n}+\cdots +q^{2kn}\right)}\right) \cdot \left( \sum_{n\geq 0}{\overline{p}(n)q^n}\right) \pmod{4}\nonumber\\
  	 &\equiv  2 \sum_{n\geq 1}{\left(q^{2n}+q^{4n}+q^{6n}+ \cdots +q^{2kn}\right)} +\sum_{n\geq 0}{\overline{p}(n) q^n} \pmod{4}.
 	\end{alignat}
 	Thus, for all $k\geq 0, n\geq 0$,
 	$$\overline{pl}_{2k+1}(2n+1)\equiv \overline{p}(2n+1)\pmod{4},$$
 	as desired.	 
 \end{proof}
\begin{corollary}\label{CoroPl1}
The following holds for every integer $n\geq 0,$
\begin{align}
	\overline{pl}(2n+1)\equiv \overline{p}(2n+1)\pmod{4}.
\end{align}
\end{corollary}
\begin{proof}
Note that for every integer $n\geq 1$, the plane overpartitions of $n$ have at most $n$ rows. Thus, we obtain for any $k\geq n$,
$$\overline{pl}(n)=\overline{pl}_{k}(n).$$
By Theorem \ref{th22}, for $k\geq n,$
\begin{align*}
\overline{pl}(2n+1)=\overline{pl}_{2k+1}(2n+1)\equiv \overline{p}(2n+1)\pmod{4}.
\end{align*}
\end{proof}

 Next result gives an infinite family of restricted plane overpartitions congruences modulo $4$.
\begin{corollary}\label{cor1}
	For all  $k, n\geq 0,$ and $\alpha\geq 0$, 
	\begin{align*}
	\overline{pl}(9^{\alpha}(54n + 45))\equiv \overline{pl}_{2k+1}(9^{\alpha}(54n + 45))\equiv 0 \pmod{4}.
	\end{align*}
\end{corollary}
\begin{proof}
	Recall that in \cite{hirschhorn}, Hirschhorn and Sellers show that $9^{\alpha}(27n + 18)$ is nonsquare for all $\alpha, n\geq 0$. Thus by \eqref{NonSqMod4}, we have $\overline{p}(9^{\alpha}(27n + 18))\equiv 0 \pmod{4}$ for all $n\geq 0$ and $\alpha\geq 0$. For any odd integer $n$, $9^{\alpha}(27n + 18)$ is odd. Replacing the odd integer $n$ by $2n+1$, the result follows by Theorem \ref{th22} and Corollary \ref{CoroPl1}.
\end{proof}
\begin{proof}[\textbf{Proof of Theorem \ref{th333}}]
Recall from the proof of Theorem \ref{th22} and \eqref{eqth11} that
\begin{align*}
	\sum_{n=0}^{\infty}{\overline{pl}_{2k+1}(n)q^{n}}\equiv  2 \sum_{n\geq 1}{\left(q^{2n}+q^{4n}+q^{6n}+ \cdots +q^{2kn}\right)} +\sum_{n\geq 0}{\overline{p}(n) q^n} \pmod{4}.
	\end{align*}
We note for odd $r>1$, then $2r\nmid \ell m+2^j$, as well  $2^{i}\nmid \ell m+2^j$ for $i>j$. Thus, we get for all $m\geq 1$, the term $q^{\ell m+2^j}$ will occur in the series 
$$2\sum_{n\geq 1}{\left(q^{2n}+q^{4n}+ q^{6n}+\cdots +q^{4kn}\right)},$$ 
only when $n=\ell m+2^j/2,\ell m+2^j/4, \ell m+2^j/8,\ldots,\ell m+2^j/2^j,$ arising from the terms $q^{2n}, q^{4n},q^{8n},\ldots,q^{2^j n},$ respectively. Thus, the coefficient of $q^{\ell m+2^j}$ in the above series is $2\sum_{i=1}^{j}1=2j\equiv 0\pmod{4}$ since $j\equiv 0\pmod{2}$. Therefor, for all $n\geq 1$,
	$$\overline{pl}_{2k+1}(\ell n+2^j)\equiv \overline{p}(\ell n+2^j)\pmod{4},$$
as desired for \eqref{th333eq1}. To prove \eqref{th333eq2}, since $k\equiv 0\pmod{2}$, we replace $k$ by $2k$ in  \eqref{eqth11} to obtain
\begin{align}
	\sum_{n=0}^{\infty}{\overline{pl}_{4k+1}(n)q^{n}}\equiv  2 \sum_{n\geq 1}{\left(q^{2n}+q^{4n}+ \cdots +q^{4kn}\right)} +\sum_{n\geq 0}{\overline{p}(n) q^n} \pmod{4}.
\end{align}
Note that for all $m\geq 1$, the term $q^{\ell m}$ will occur in the series 
$$2\sum_{n\geq 1}{\left(q^{2n}+q^{4n}+ \cdots +q^{4kn}\right)},$$ 
when $n=\ell m/2,\ell m/4,\ldots,\ell m/4k,$ arising from the terms $q^{2n}, q^{4n},\ldots,q^{4kn},$ respectively. Thus, the coefficient of $q^{\ell m}$ in the above series is $2\sum_{i=1}^{2k}1=4k\equiv 0\pmod{4}$. Therefor, for all $n\geq 0$,
	$$\overline{pl}_{4k+1}(\ell n)\equiv \overline{p}(\ell n)\pmod{4},$$
where $\ell$ here is the least common multiple of all even positive integers $\leq 4k.$
\end{proof}

\begin{proof}[\textbf{Proof of Theorem \ref{ThmMod8}}]
	Observe that by \eqref{PLOvkGen} and Lemma \ref{lem1}, we have that
	\begin{align*}
	\sum_{n=0}^{\infty}{\overline{pl}_{4}(n)q^{n}}
	&= f(q) f(q^2)^2 f(q^3)^3  \prod_{n\geq 4}{f(q^n)^4} & &\\
	&\equiv f(q) f(q^2)^2 f(q^3)^3\pmod{8} & &\\ 
	& \equiv \left( 1+2\sum_{n\geq 1}{q^{n}}\right)\cdot \left( 1+2\sum_{n\geq 1}{q^{2n}}\right)^2\cdot \left( 1+2\sum_{n\geq 1}{q^{3n}}\right)^3 \pmod{8}&&
	\end{align*}
Thus,
	\begin{align}\label{gen2}
	\sum_{n=0}^{\infty}{\overline{pl}_{4}(n)q^{n}}&\equiv 1+2\sum_{n\geq 1}{q^n}+4\sum_{n\geq 1}{q^{2n}} +6\sum_{n\geq 1}{q^{3n}}+&&\\
		&\;\;\;\;\;4\sum_{m,n\geq 1}{q^{2(n+m)}}+4\sum_{m,n\geq 1}{q^{3(n+m)}}+4\sum_{m,n\geq 1}{q^{n+3m}}\pmod{8}&&\nonumber
	\end{align}
	For any $k\geq 1$, the term $q^{12k}$ will occur in the series
	$$\sum_{n\geq 1}{q^n}, \sum_{n\geq 1}{q^{2n}}, \sum_{n\geq 1}{q^{3n}}$$
	when $n=12k, 6k, 4k$, arising from the terms $q^{n}, q^{2n},q^{3n}$, respectively. Also, the term $q^{12k}$ will occur in the series
	\begin{align}\label{se1}
    \sum_{m,n\geq 1}{q^{2(n+m)}}, \sum_{m,n\geq 1}{q^{3(n+m)}}, \sum_{m,n\geq 1}{q^{n+3m}}
	\end{align}
	when $n+m=6k, 4k$ and $n+3m=12k$, arising from the terms $q^{2(n+m)}, q^{3(n+m)},q^{n+3m}$, respectively. We use Lemma \ref{lemma(ab)} to count the appearances of $q^{12k}$ in the three series of \eqref{se1} and catalog the results in the following table.
	
	\FloatBarrier
	\begin{table}[H]
	\caption{Coefficients of $q^{12k}$ in the series of \eqref{se1}} 
	\centering 
	\begin{tabular}{c c c c c c} 
	\hline\hline 
	$an+bm=abc$ & $a$ & $b$ & $c$& $j$& coefficient of $q^{abcj}$ in $\sum_{n,m\geq 1}q^{j(an+bm)}$\\ [0.5ex] 
	\hline 
	
	$n+m=6k$    & $1$ & $1$ & $6k$ & $2$ & $6k-1$ \\ 
	$n+m=4k$    & $1$ & $1$ & $4k$ & $3$ & $4k-1$ \\
	$n+3m=12k$  & $1$ & $3$ & $4k$ & $1$ & $4k-1$ \\[1ex] 
	\hline 
	\end{tabular}
	\label{table1} 
	\end{table}
	
	\noindent Thus by Table \ref{table1}, the coefficient of $q^{12k}$ in the series on the right hand side of \eqref{gen2} is 
	$$2+4+6+4(6k-1)+4(4k-1)+4(4k-1)\equiv 0\pmod{8},$$
	which proves (\ref{c11}).\\
	To prove (\ref{c22}), we observe that for any $k\geq 1$, the term $q^{6k+3}$ will occur in the series
	$$\sum_{n\geq 1}{q^n}, \sum_{n\geq 1}{q^{3n}}$$
	from \eqref{gen2} when $n=6k+3, 2k+1,$ arising from the terms $q^{n}, q^{3n}$, respectively. Also, the term $q^{6k+3}$ will occur in the series
	\begin{align}\label{ss2}
    \sum_{m,n\geq 1}{q^{3(n+m)}}, \sum_{m,n\geq 1}{q^{n+3m}}
	\end{align}
	from \eqref{gen2} when $n+m=2k+1, n+3m=6k+3$ arising from the terms $q^{3(n+m)},q^{n+3m}$ respectively. However, the term $q^{6k+3}$ does not occur in the series 
	\begin{align}
	 \sum_{m,n\geq 1}{q^{2n}}, \sum_{m,n\geq 1}{q^{2(n+m)}},
	\end{align}
	because $6k+3$ is not divisible by $2$ for every integer $k\geq 1$.	Again, we use Lemma \ref{lemma(ab)} to conclude the number of occurrences of $q^{6k+3}$ in the series of \eqref{ss2} in the following table

		\FloatBarrier
		\begin{table}[H]
		\caption{Coefficients of $q^{6k+3}$ in the series of \eqref{ss2}} 
		\centering 
		\begin{tabular}{c c c c c c} 
		\hline\hline 
		$an+bm=abc$ & $a$ & $b$ & $c$& $j$& coefficient of $q^{abcj}$ in $\sum_{n,m\geq 1}q^{j(an+bm)}$\\ [0.5ex] 
		\hline 
		
		$n+m=2k+1$    & $1$ & $1$ & $2k+1$ & $3$ & $2k$ \\ 
		$n+3m=6k+3$   & $1$ & $3$ & $2k+1$ & $1$ & $2k$ \\[1ex] 
		\hline 
		\end{tabular}
		\label{table6k+3} 
		\end{table}
	
\noindent Thus by Table \ref{table6k+3}, the coefficient of  $q^{6k+3}$ in the series on the right hand side of \eqref{gen2} is
	$$2+6+4\cdot 2k+4\cdot 2k\equiv 0\pmod{8},$$
	which proves (\ref{c22}).
	
	We now prove (\ref{c34}) while (\ref{c33}) can be proved similarly with less effort. We observe that by \eqref{PLOvkGen} and Lemma \ref{lem1} that
	\begin{align*}
	\sum_{n=0}^{\infty}{\overline{pl}_{8}(n)q^{n}}
	&= f(q)f(q^2)^2 f(q^3)^3 f(q^4)^4 f(q^5)^5 f(q^6)^6 f(q^7)^7 \prod_{n\geq 8}{f(q^n)^8} & &\\
	&\equiv f(q)\; f(q^2)^2\; f(q^3)^3\;  f(q^5)\; f(q^6)^2\; f(q^7)^3 \pmod{8} &&\\ 
	& \equiv \left( 1+2\sum_{n\geq 1}{q^{n}}\right)\cdot \left( 1+2\sum_{n\geq 1}{q^{2n}}\right)^2\cdot \left( 1+2\sum_{n\geq 1}{q^{3n}}\right)^3 \cdot &&\\
	&\;\;\;\;\left(1+2\sum_{n\geq 1}{q^{5n}}\right)\cdot \left( 1+2\sum_{n\geq 1}{q^{6n}}\right)^2\cdot \left( 1+2\sum_{n\geq 1}{q^{7n}}\right)^3\pmod{8}.&&
	\end{align*}
	Thus we have
	\begin{align*}
	\sum_{n=0}^{\infty}{\overline{pl}_{8}(n)q^{n}}
	&\equiv 1+2\sum_{n\geq 1}{q^n}+4\sum_{n\geq 1}{q^{2n}} +6\sum_{n\geq 1}{q^{3n}}+2\sum_{n\geq 1}{q^{5n}}+4\sum_{n\geq 1}{q^{6n}}+6\sum_{n\geq 1}{q^{7n}}+&&\\
	&\;\;\;\;\;4\sum_{m,n\geq 1}{q^{2(n+m)}}+4\sum_{m,n\geq 1}{q^{3(n+m)}}+4\sum_{m,n\geq 1}{q^{6(n+m)}}+4\sum_{m,n\geq 1}{q^{7(n+m)}}+&&\\
	&\;\;\;\;4\sum_{m,n\geq 1}{q^{n+3m}}+4\sum_{m,n\geq 1}{q^{n+5m}}+4\sum_{m,n\geq 1}{q^{n+7m}}+&&\\
	&\;\;\;\;4\sum_{m,n\geq 1}{q^{3n+5m}}+4\sum_{m,n\geq 1}{q^{3n+7m}}+4\sum_{m,n\geq 1}{q^{5n+7m}}\pmod{8}.&&
	\end{align*}
	
	For any $k\geq 1$, the term $q^{210k+3}$ will occur in the series
	$$\sum_{n\geq 1}{q^n}, \sum_{n\geq 1}{q^{3n}}$$
	when $n=210k+3, 70k+1$ arising from the terms $q^{n}, q^{3n}$ respectively. Also, the term $q^{210k+3}$ will occur in the series
	\begin{align*}
	&\sum_{m,n\geq 1}{q^{3(n+m)}},\sum_{m,n\geq 1}{q^{n+3m}},\sum_{m,n\geq 1}{q^{n+5m}}\\
	&\sum_{m,n\geq 1}{q^{n+7m}},\sum_{m,n\geq 1}{q^{3n+5m}},\sum_{m,n\geq 1}{q^{3n+7m}},\sum_{m,n\geq 1}{q^{5n+7m}},
	\end{align*}
	when $ 3(n+m),n+3m, n+5m,n+7m, 3n+5m, 3n+7m,5n+7m=210k+3$ arising from the terms $$ q^{3(n+m)}, q^{n+3m},q^{n+5m}, q^{n+7m}, q^{3n+5m},  q^{3n+7m}, q^{5n+7m}$$ respectively. Since $210k+3$ is not divisible by $2,5,6,7,$ so the term $q^{210k+3}$ will not occur in any of the following $q$-series,
	$$\sum_{n\geq 1}{q^{2n}},\sum_{n\geq 1}{q^{5n}},\sum_{n\geq 1}{q^{6n}},\sum_{n\geq 1}{q^{7n}},\sum_{m,n\geq 1}{q^{2(n+m)}},\sum_{m,n\geq 1}{q^{6(n+m)}},\sum_{m,n\geq 1}{q^{7(n+m)}} .$$
\noindent	Again, by applying Lemma \ref{lemma(ab)}, the appearances of $q^{210k+3}$ in the series
	\begin{align}\label{ss33}
	\sum_{m,n\geq 1}{q^{3(n+m)}},\sum_{m,n\geq 1}{q^{n+3m}},
	\end{align}
	are given in the following table.
	
			\FloatBarrier
			\begin{table}[H]
			\caption{Coefficients of $q^{210k+3}$ in the series of \eqref{ss33}} 
			\centering 
			\begin{tabular}{c c c c c c} 
			\hline\hline 
			$an+bm=abc$ & $a$ & $b$ & $c$& $j$& coefficient of $q^{abcj}$ in $\sum_{n,m\geq 1}q^{j(an+bm)}$\\ [0.5ex] 
			\hline 
			
			$n+m=70k+1$    & $1$ & $1$ & $70k+1$ & $3$ & $70k$ \\ 
			$n+3m=210k+3$   & $1$ & $3$ & $70k+1$& $1$ & $70k$ \\[1ex] 
			\hline 
			\end{tabular}
			\label{table21} 
			\end{table}

	  
	  Now for $n+5m, n+7m=210k+3$, we have the following enumerations
	$$5\cdot 1+(210k-5+3),5\cdot 2+(210k-10+3),\dots, 5\cdot 42k+3,$$
	$$7\cdot 1+(210k-7+3),7\cdot 2+(210k-14+3),\dots, 7\cdot 30k+3.$$ 
	Thus, we have $42k, 30k$ pairs of $m$ and $n$ for $n+5m, n+7m=210k+3$, respectively. 
	
	For $3n+5m=210k+3$, then $5m=210k+3-3n$ and so $3$ divides $m$. Thus, counting for $3n+5m=210k+3$ is equivalent to counting for $3n+15m=210k+3$ which is equivalent to $n+5m=70k+1$ and the later has the following enumerations
		$$5\cdot 1+(70k-5+3),5\cdot 2+(70k-10+3),\dots, 5\cdot 14k+3.$$
	 Hence, we obtain $14k$ possible pairs $n$ and $m$ such that  $3n+5m=210k+3$. Similarly, we have $10k$ pairs of positive integers $m$ and $n$ such that $3n+7m=210k+3$.
	 Thus, the following table catalogs the coefficients of the term  $q^{210k+3}$ in the following series
	 \begin{align}\label{ss44}
     \sum_{m,n\geq 1}{q^{n+5m}},\sum_{m,n\geq 1}{q^{n+7m}},\sum_{m,n\geq 1}{q^{3n+5m}},\sum_{m,n\geq 1}{q^{3n+7m}}.
	 \end{align}

	 			\FloatBarrier
	 			\begin{table}[H]
	 			\caption{Coefficients of $q^{210k+3}$ in the series of \eqref{ss44}} 
	 			\centering 
	 			\begin{tabular}{c c c c} 
	 			\hline\hline 
	 			$an+bm=210k+3$ & $a$ & $b$ &   coefficient of $q^{210k+3}$ in $\sum_{n,m\geq 1}q^{an+bm}$\\ [0.5ex] 
	 			\hline 
	 			
	 			$n+5m=210k+3$    & $1$ & $5$ &   $42k$\\
	 			$n+7m=210k+3$    & $1$ & $7$ &   $30k$\\
	 			$3n+5m=210k+3$   & $3$ & $5$ &   $14k$\\
	 			$3n+7m=210k+3$   & $3$ & $7$ &   $10k$ \\[1ex] 
	 			\hline 
	 			\end{tabular}
	 			\label{table321} 
	 			\end{table}

	 Now, we only need to check the coefficient of $q^{210k+3}$ in the series $\sum_{m,n\geq 1}{q^{5n+7m}}$.
	Note that the integers $a=5, b=7$ and $c=210k+3$ satisfy the desired conditions of Lemma \ref{LemPart}.  Thus $M_{210k+3}$ is the number of the possible pairs of positive integers $(n,m)$ such that $5n+7m=210k+3$ and 
	$$M_{210k+3}\equiv M_{210k+3+\pi_{8}}\pmod{8},$$
	where $\pi_{8}$ is the minimum period modulo $8$ of the following $q$-series
	 \begin{align*}
	 A(q) :=\sum_{n=0}^{\infty}{p(n;S)q^{n}}=\frac{1}{(1-q^5)(1-q^7)}.
	 \end{align*}
	  Letting $S=\{5,7\}$, $\ell=2$, and $N=3$ in Theorem \ref{kwong}, then $\pi_{8}=\pi_{8}(A)=280$. In other words, for all $n\geq 0$,
	 \begin{align*}
      M_{210k+3+\pi_{8}}=p(210k+3+\pi_{8}(A);S)\equiv p(210k+3;S)=M_{210+3} \pmod{8}.
	 \end{align*} 
	 If we let $k=4j$ where $j\in \mathbb{N}$, then we observe by the periodicity of $A(q)$ that
	 $$M_{210k+3}=p(210k+3;S)=p(3+3j\cdot \pi_{8}(A);S)\equiv p(3;S)\pmod{8}=0\pmod{8}.$$
	 By a similar argument for $k=4j-1,4j-2,4j-3$, we obtain the following
	 	\[   
	 	\ M_{210k+3} = p(210k+3;S)\equiv
	 	\begin{cases}
	 	p(3;S)\;\;\;\;\pmod{8}=0\pmod{8} &\text{if $k=4j=4,8,12,\dots$}\\
	 	p(73;S)\;\;\pmod{8}=2\pmod{8} &\text{if $k=4j-1=3,7,11,\dots$}\\
	 	p(143;S)\pmod{8}=4\pmod{8} &\text{if $k=4j-2=2,6,10,\dots$}\\ 
	 	p(213;S)\pmod{8}=6\pmod{8} &\text{if $k=4j-3=1,5,9,\dots$}\
	 	\end{cases}
	 	\] 
	By summing all coefficients of $q^{210k+3}$ and using Tables \ref{table21} and \ref{table321}, we get
	$$2+6+4\cdot (70k+70k+42k+30k+14k+10k+M_{210k+3})\equiv 0 \pmod{8},$$
	which proves (\ref{c34}).  Similarly, the identities \eqref{c35} and \eqref{c37} can be proved using the same technique. However, for the sake of completeness, we show in Tables  \eqref{table210k+9} and \eqref{table210+105} the corresponding coefficients of the terms $q^{210k+9}, q^{210k+105}$ modulo $8$ of the generating function of $8$-rowed plane overpartitions, $\sum_{n=0}^{\infty}{\overline{pl}_{8}(n)q^n}$. 
    
			\FloatBarrier
			\begin{table}[H]
			\caption{Coefficients of $q^{210k+9}, q^{210k+105}$ modulo $8$ in $c\sum_{n\geq 1}{q^{jn}}$  } 
			\centering 
			\begin{tabular}{c c c c c } 
			\hline\hline 
			$c\sum_{n\geq 1}q^{jn}$ & coefficient of& $q^{210k+9},$& $q^{210k+105}$ &in $c\sum_{n\geq 1}q^{jn}$\\ [0.5ex] 
			\hline 
			
			$2\sum_{n\geq 1}q^n$            & &  $2$           &$2$&\\ 
			$4\sum_{n\geq 1}q^{2n}$          &&  $0$           &$0$&\\
			$6\sum_{n\geq 1}q^{3n}$          &&  $6$           &$6$&\\
			$2\sum_{n\geq 1}q^{5n}$          &&  $0$           &$2$&\\
			$4\sum_{n\geq 1}q^{6n}$          &&  $0$           &$0$&\\
			$6\sum_{n\geq 1}q^{7n}$          &&  $0$           &$6$&\\[1ex]
			
			\hline 
			
			\end{tabular}
			\label{table210k+9} 
			\end{table}

			\FloatBarrier
			\begin{table}[H]
			\caption{Coefficients of $q^{210k+9}, q^{210k+105}$ modulo $8$ in $c\sum_{n,m \geq 1}{q^{j(an+bm)}}$  } 
			\centering 
			\begin{tabular}{c c c c c } 
			\hline\hline 
			$c\sum_{n,m\geq 1}q^{j(an+bm)}$ & coefficient of& $q^{210k+9},$& $q^{210k+105}$ &in $c\sum_{n,m\geq 1}q^{j(an+bm)}$\\ [0.5ex] 
			\hline 
			
			$4\sum_{n,m\geq 1}q^{2(n+m)}$    &&  $0$           &$0$&\\
			$4\sum_{n,m\geq 1}q^{3(n+m)}$    &&  $4(70k+2)$    &$4(70k+34)$&\\
			$4\sum_{n,m\geq 1}q^{6(n+m)}$    && $0$            &$0$&\\
			$4\sum_{n,m\geq 1}q^{7(n+m)}$    && $0$            &$4(30k+14)$&\\
			$4\sum_{n,m\geq 1}q^{n+3m}$      && $4(70k+2)$     &$4(70k+34)$&\\
			$4\sum_{n,m\geq 1}q^{n+5m}$      && $4(42k+1)$     &$4(42k+20)$&\\
			$4\sum_{n,m\geq 1}q^{n+7m}$      &&  $4(30k+1)$    &$4(10k+14)$&\\
			$4\sum_{n,m\geq 1}q^{3n+5m}$     &&  $4\cdot 14k$  &$4(14k+6)$&\\
			$4\sum_{n,m\geq 1}q^{3n+7m}$     &&  $4\cdot 10k$  &$4(10k+4)$&\\
			$4\sum_{n,m\geq 1}q^{5n+7m}$     &&  $4\cdot M_{210k+9}$ &$4(6k+2)$&\\[1ex] 
			\hline 
			
			\end{tabular}
			\label{table210+105} 
			\end{table}

\indent Now, we only need to show that the number $M_{210k+9}$ is even.
Similar to the argument above and by applying Lemma \ref{LemPart} for $a=5,b=7$ and $c=210k+9$, we obtain the following,

\[   
	 	\ M_{210k+9} = p(210k+9;S)\equiv
	 	\begin{cases}
	 	p(9;S)\;\;\;\;\pmod{8}=0\pmod{8} &\text{if $k=4j=4,8,12,\dots$}\\
	 	p(79;S)\;\;\pmod{8}=2\pmod{8} &\text{if $k=4j-1=3,7,11,\dots$}\\
	 	p(149;S)\pmod{8}=4\pmod{8} &\text{if $k=4j-2=2,6,10,\dots$}\\ 
	 	p(219;S)\pmod{8}=6\pmod{8} &\text{if $k=4j-3=1,5,9,\dots$}\
	 	\end{cases}
	 	\] 
Thus by $M_{210k+9}$ is being even, the coefficient of $q^{210k+9}$ modulo $8$ in  $\sum_{n=0}^{\infty}{\overline{pl}_{8}(n)q^n}$ is given by summing all coefficients of $q^{210k+9}$ in Tables \eqref{table210k+9} and \eqref{table210+105}, so we obtain
$$2+6+4\cdot \left(70k+2+70k+2+42k+1+30k+1+14k+10k+M_{210k+9}\right)\equiv 0 \pmod{8},$$
as desired for the identity \eqref{c35}.
 
For the identity \eqref{c37}, by Tables \eqref{table210k+9} and \eqref{table210+105}, the corresponding coefficient modulo $8$ of $q^{210k+105}$  in the series $\sum_{n=0}^{\infty}{\overline{pl}_{8}(n)q^n}$ is congruent to
\begin{align*}
16+4(252k+128)\equiv 0\pmod{8}.
\end{align*} 
  
\end{proof}

 \begin{proof}[\textbf{Proof of Theorem \ref{CongOverMod8}}]
  We recall from \eqref{GenConOver} that 
  \begin{align*}
  \sum_{n=0}^{\infty}\overline{p}(n)q^{n}\equiv \prod_{j=0}^{k-2}\left(\phi(q^{2^j})\right)^{2^j}\pmod{2^{k}},
  \end{align*}
  where as in \eqref{mo1},
  \begin{align*}
  \phi(q)=\sum_{n=-\infty}^{\infty}{q^{n^2}}=1+2\sum_{n=1}^{\infty}{q^{n^2}}.
  \end{align*}
  For the case $k=3$,
  \begin{align}\label{GenOvCongaMod8}
  \sum_{n=0}^{\infty}\overline{p}(n)q^{n}\equiv\phi(q)\cdot \phi(q^2)^2 \equiv 1+2\sum_{n\geq 1}{q^{n^2}}+4\sum_{n\geq 1}{q^{2n^2}}+4\sum_{n,m\geq 1}{q^{2(n^2+m^2)}}\pmod{8}.
  \end{align}
 If $n$ is a nonsquare odd integer, then $n$ can not be written as $m^2,2m^2,$ or $2(m^2+k^2)$ for all $m,k\geq 1$. Thus by \eqref{GenOvCongaMod8}, the result follows.
 \end{proof}
 
 As a consequence, we obtain the following result which gives an infinite family of overpartition congruences modulo $8$.
 \begin{corollary}\label{infam}
 For any integer $\alpha\geq 3,$ and $\beta\geq 0,$ the following holds for each $n\geq 0,$
 \begin{align*}
 \overline{p}(2^\alpha 3^\beta n+5)\equiv 0\pmod{8}.
 \end{align*}
 \end{corollary}
 \begin{proof}
 Clearly, for $\alpha\geq 3,$ and $\beta\geq 0,$ we have that $ 2^\alpha 3^\beta n+5$ is an odd integer for each $n\geq 0$. Suppose that there is a positive integer $m$ such that $2^\alpha 3^\beta n+5=(2m+1)^2$. Thus, we obtain $2^{\alpha-2} 3^\beta n+1=m(m+1)$. We know that $m(m+1)$ is even which contradicts the fact $2^{\alpha-2} 3^\beta n+1$ is odd since $\alpha-2\geq 1$. Thus no such $m$ exists, and $2^\alpha 3^\beta n+5$ is not an odd square.
 \end{proof}
 
 Next, we obtain a result of  Hirschhorn and Sellers \cite{hirschhorn2005arithmetic}.
 \begin{corollary}\label{OvCon4n+3}
 The following holds for all $n\geq 0,$
 \begin{align}
 \overline{p}(4n+3)\equiv 0\pmod{8}.
 \end{align}
 \end{corollary}
 \begin{proof}
 Similar to the proof of Corollary \ref{infam}, $4n+3$ is a nonsquare odd integer for all $n\geq 0$.
 \end{proof}
 

\begin{proof}[\textbf{Proof of Theorem \ref{5rowed}}]
By Lemma \ref{lem1} and the fact \eqref{OverPartMod2}, $4\overline{p}(n)\equiv 0\pmod{8}$ for every integer $n\geq 1$,
\begin{align*}
	\sum_{n=0}^{\infty}{\overline{pl}_{5}(n)q^{n}}&=f(q) f(q^2)^2 f(q^3)^3 f(q^4)^4  \prod_{n\geq 5}{f(q^n)^5}\\
	&\equiv f(q^2) \; f(q^3)^2 \; f(q^4)^3 \; \prod_{n=1}^{\infty}f(q^n)\pmod{8} \\ 
	& \equiv \left( 1+2\sum_{n\geq 1}{q^{2n}}\right)\left( 1+2\sum_{n\geq 1}{q^{3n}}\right)^2 \left( 1+2\sum_{n\geq 1}{q^{4n}}\right)^3 \left(1+\sum_{n\geq 1}{\overline{p}(n)q^n}\right)\pmod{8}\\
	&\equiv 1+2\sum_{n\geq 1}{q^{2n}}+4\sum_{n\geq 1}{q^{3n}}+6\sum_{n\geq 1}{q^{4n}}\\ &\;\;\;\;\;\;+4\sum_{n,m\geq 1}{q^{3(n+m)}}+4\sum_{n,m\geq 1}{q^{4(n+m)}}+4\sum_{n,m\geq 1}{q^{2n+4m}}\\
	&\;\;\;\;\;\;+\sum_{n\geq 1}{\overline{p}(n)q^n}+2\sum_{n,m\geq 1}{\overline{p}(n)q^{n+2m}}+6\sum_{n,m\geq 1}{\overline{p}(n)q^{n+4m}}\pmod{8}.
\end{align*}
We observe that 
\begin{align*}
2\sum_{n,m\geq 1}{\overline{p}(n)q^{n+2m}}+6\sum_{n,m\geq 1}{\overline{p}(n)q^{n+4m}}&=2\sum_{n,m\geq 1}{\overline{p}(n)q^{n+4m-2}}+8\sum_{n,m\geq 1}{\overline{p}(n)q^{n+4m}}\\
&\equiv 2\sum_{n,m\geq 1}{\overline{p}(n)q^{n+4m-2}}\pmod{8}.
\end{align*}
Thus, we obtain
\begin{align}\label{EqPL5}
	\sum_{n=0}^{\infty}{\overline{pl}_{5}(n)q^{n}}
	&\equiv 1+\sum_{n\geq 1}{\left(2q^{2n}+4q^{3n}+6{q^{4n}}\right)}+4\sum_{n,m\geq 1}{\left(q^{3(n+m)}+q^{4(n+m)}+q^{2(n+2m)}\right)}\nonumber\\
	&\;\;\;\;+\sum_{n\geq 1}{\overline{p}(n)q^n}+2\sum_{n,m\geq 1}{\overline{p}(n)q^{n+4m-2}}\pmod{8}.
\end{align}
Note that $12k+1$ is not divisible by $2$,$3$ and $4$. So for any $k\geq 1$, the term $q^{12k+1}$ will occur only in the series
	$$\sum_{n\geq 1}{\overline{p}(n)q^n}, \sum_{n,m\geq 1}{\overline{p}(n)q^{n+4m-2}},$$
when $n=12k+1, 12k+1-(4m-2)$ arising from the terms $q^n, q^{12k+1-(4m-2)}$ respectively for $m=1,\dots,3k.$ Hence, the coefficient of $q^{12k+1}$ in the series on the right hand side of \eqref{EqPL5} is then given by
\begin{align*}
\overline{p}(12k+1)+2\sum_{m=1}^{3k}\overline{p}(12k-4m+3).
\end{align*}
Note that by Corollary \ref{OvCon4n+3}, for all $k\geq 1$ and  $m=1,\dots,3k$, we have
 $$\overline{p}(12k-4m+3)=\overline{p}(4(3k-m)+3)\equiv 0\pmod{8}.$$  
 Thus,
\begin{align*}
\sum_{m=1}^{3k}\overline{p}(12k-4m+3)\equiv 0\pmod{8}.
\end{align*} 
Therefore, for all $k\geq 1$,
\begin{align*}
\overline{pl}_{5}(12k+1)\equiv \overline{p}(12k+1)+ \sum_{m=1}^{3k}\overline{p}(12k-4m+3)\equiv\overline{p}(12k+1)\pmod{8}.
\end{align*}
For the case $k=0,$
$$\overline{pl}_{5}(1)\equiv \overline{p}(1)\pmod{8}.$$ 
Thus, for every integer $k\geq 0$,
\begin{align*}
\overline{pl}_{5}(12k+1)\equiv 	\overline{p}(12k+1)\pmod{8},
\end{align*}
as desired for \eqref{Con1Pl5}.
The congruence \eqref{ConPl5} can be proved similarly.  
\end{proof}

We lastly end this section by combining Theorem \ref{5rowed} and Corollary \ref{infam} to obtain the following infinite family of $5$-rowed plane overpartition congruences modulo $8$.
\begin{corollary}\label{LastThm}
For any integers $\alpha \geq 3$ and $\beta \geq 1$, the following holds for all $n\geq 0,$
	\begin{align}
	\overline{pl}_{5}(2^\alpha 3^\beta n+5)\equiv 0\pmod{8}.
	\end{align}
\end{corollary}
\begin{proof}
Note that by Theorem \ref{5rowed}, for all $n\geq 0,$
\begin{align*}
\overline{pl}_{5}(2^\alpha 3^\beta n+5)=\overline{pl}_{5}(12(2^{\alpha-2} 3^{\beta-1} n)+5)\equiv \overline{p}_{5}(12(2^{\alpha-2} 3^{\beta-1} n)+5)\pmod{8}.
\end{align*}
The rest follows by Corollary \ref{infam}.
\end{proof}

\section{\textbf{Concluding Remarks}}\label{Remark}

We close this paper with a few comments and ideas. We established several examples of plane and restricted plane overpartition congruences modulo $4$ and $8$. Often, our technique is based on applying Lemma \ref{lem1} up to a small power of $2$, then collecting the coefficients of certain terms of the desired power. Lemma \ref{lem1} can be a very powerful tool to find and prove additional congruences modulo powers of $2$ for any partition function that involves products containing functions of the form $f(q^n)^m$ where $f$ is defined by $f(q)=\frac{1+q}{1-q}$. For example, the overpartition function has this property.
 
Based on computational evidence, we conjecture that for each integer $r\geq 1,$ and each $k\geq 1$, there exist infinitely many integers $n$ such that
\begin{align}\label{tackle}
\overline{pl}_{k}(n)\equiv 0\pmod{2^r}.
\end{align}
If this holds, then for infinitely many integers $n$,
\begin{align}
\overline{pl}(n)\equiv 0\pmod{2^r}.
\end{align}

Lemma \ref{lem1} might be a powerful tool to tackle such congruences as \eqref{tackle}. We  note that Theorems \ref{th22} and \ref{5rowed} suggest there might be other arithmetic relations between plane overpartitions and overpartitions  that are worth investigating. Furthermore, computational evidence suggests that there is a relation modulo powers of $2$ between overpartitions and restricted plane overpartitions. Thus, we conjecture that for each  $r\geq 1$ and each $k\geq 1$, there exist infinitely many integers $n$, such that
\begin{align*}
\overline{pl}_{k}(n)\equiv \overline{p}(n) \pmod{2^r}.
\end{align*}

Another approach to establish congruences for plane overpartitions modulo powers of $2$ is to look for an iteration formula for plane overpartitions similar to that of overpartitions  given by Theorem $2.2$ of \cite{hirschhorn2006arithmetic}. That is, consider
\begin{align*}
\overline{P}(q)=\phi(q)\;\phi^2(q^2)\;\phi^4(q^4)\;\phi^8(q^8)\cdots,
\end{align*}
and let
\begin{align*}
G_{n}(q):=\prod_{i=n+1}^{\infty}{\frac{1+q^i}{1-q^i}=\prod_{i=n+1}^{\infty}{f(q^{i})}}.
\end{align*}
Thus the generating function for plane overpartitions can be rewritten as
\begin{align*}
\overline{PL}(q)=&\overline{P}(q)\cdot G_{1}(q) \cdot G_{2}(q)\cdot G_{3}(q) \cdots\\
&=\prod_{n=1}^{\infty}{\phi(q^{2^{n-1}})^{2^{n-1}} G_{n}(q).}
\end{align*}
Investigating properties of $G_{n}(q)$ might yield congruences modulo higher powers of $2$ for plane overpartitions.

\section{\textbf{Acknowledgements}}
This work is a part of my PhD thesis written at Oregon State University. I would like to express special thanks of gratitude to my advisor Professor Holly Swisher for her guidance and helpful suggestions. Also, I would like to thank Professor James Sellers for helpful discussions and encouragement that motivated this work. 

\bibliographystyle{plain}
\bibliography{ref}

\end{document}